\documentclass[12pt]{article}
\usepackage{amsmath,amssymb,amsfonts}
\usepackage{bm}
\usepackage{graphicx}
\usepackage[margin=1in]{geometry}
\usepackage[hidelinks]{hyperref}

\newtheorem{lemma}{Lemma}
\newtheorem{proposition}{Proposition}
\newtheorem{theorem}{Theorem}
\newtheorem{corollary}{Corollary}

\title{Wavelet covariances under fixed finite observations:\\
dilation symmetry breaking and distinct endpoint limits}
\author{Jun Tsuzurugi\\
Department of Information Science and Engineering,\\
Okayama University of Science, Japan\\
\texttt{juntuzu@ous.ac.jp}}
\date{September 2026}

\begin{document}
\maketitle
\begingroup
\renewcommand{\thefootnote}{}
\footnotetext{Preprint of an article submitted for consideration in International Journal of Wavelets, Multiresolution and Information Processing. \copyright\ 2026 World Scientific Publishing Company. \url{https://www.worldscientific.com/worldscinet/ijwmip}}
\endgroup

\begin{abstract}
For a centered homogeneous generalized Gaussian field with spectral weight
$C|\omega|^{-\gamma}$, we prove two-endpoint recovery of normalized
posterior wavelet covariances under fixed finite admissible observations.
Global dilation symmetry holds if and only if all observation functions
vanish.  The nonzero wavelet and the observation functions are real
Schwartz functions with $M\geq1$ vanishing moments, with $C>0$ and
$\gamma<2M+1$.  Noise is Gaussian and field-independent, and total
observation covariance is positive definite.
Under simultaneous dilation of target scales and positions, multiplication
by $\lambda^{-\gamma}$ recovers the prior covariance as
$\lambda\downarrow0$ and $\lambda\uparrow\infty$, uniformly on every
fixed compact target set.
If a nonzero observation function exists, additional finite first nonzero
Fourier orders $m_\psi,m_i$ at zero of the wavelet and nonzero observation functions
yield compact-uniform recovery bounds
$O(\lambda^{2m_\psi+1-\gamma})$ and
$O(\lambda^{-(2\min_i m_i+1-\gamma)})$, with leading coefficients and
cancellation-dependent improvements.  The bounds are exact leading
orders when their coefficients are nonzero.
For the Mat\'ern-type spectrum $S(\omega)=C(\kappa^2+\omega^2)^{-p}$,
$\kappa>0$, $1/2<p<M+1/2$, fixed finite noisy point observations yield
distinct compact-uniform limits: the homogeneous $C|\omega|^{-2p}$ kernel
after multiplication by $\lambda^{-2p}$ at small scales, and the
constant-spectrum $S(0)$ kernel without normalization at large scales.
A projection-norm criterion treats fixed closed observation subspaces;
Mexican-hat formulas illustrate the finite-observation results.
\end{abstract}

\section{Introduction}
Wavelet variances describe how fluctuations depend on scale, whereas
cross-scale covariances describe dependence between coefficients at
different scales and positions.  Spectral-filter formulas and power-law
wavelet variances are established tools in the analysis of fractal and
long-memory signals \cite{wornell1992,moulines2007}.
Coeurjolly, Amblard, and Achard obtain wavelet covariance and cross-spectral
formulas for multivariate fractional Brownian motion
\cite{coeurjolly2010}.  These prior covariance relations do not by
themselves describe the effect of holding a finite observation set fixed
while dilating the target wavelets.

Gaussian conditioning supplies the corresponding posterior covariance
as a finite-rank subtraction \cite[Eq.~(2.24)]{rasmussen2006}.
Related projection formulas for conditioning on finitely many linear
functionals appear in the theory of generalized Gaussian bridges
\cite[Theorem~3.1]{sottinen2014}.
Marcotte and Gloaguen construct exact conditioning of Gaussian fields
on discrete wavelet coefficients and hard point data, using their
covariance and cross-covariance matrices and cokriging post-conditioning
for larger images \cite{marcotte2015}.  This is a direct precedent for
conditioning in the wavelet setting.
These conditioning mechanisms are established; the question addressed
here is how the posterior correction interacts with simultaneous
dilation of wavelet scales and positions when the observations remain
fixed.
Moment asymptotic expansions of wavelet transforms at large and small
scales are also established \cite{pathak2014}.  Our quantitative result
uses this Taylor/moment mechanism to determine the normalized posterior
correction, rather than proposing a new expansion method.

The first object is a centered homogeneous generalized Gaussian field
on the test space $\mathcal S_M$, with spectral covariance weight
$C|\omega|^{-\gamma}$ and $\gamma<2M+1$.  An ordinary finite-variance
stationary process cannot have this pure power-law spectrum at every
frequency.  The ordinary-process formulas below fix the wavelet and
Fourier conventions; the field construction in Section~\ref{sec:posterior}
provides the setting for the homogeneous-field results.
For fixed finite admissible observations with field-independent Gaussian
noise and positive-definite total observation covariance,
Proposition~\ref{prop:posterior-recovery} proves that the posterior
covariance multiplied by $\lambda^{-\gamma}$ converges to the prior
covariance as $\lambda\downarrow0$ and as $\lambda\uparrow\infty$,
uniformly on each fixed compact set of target pairs.
Corollary~\ref{cor:broken-symmetry} gives the
corresponding necessary and sufficient condition for exact symmetry:
the posterior obeys the original dilation identity for all target pairs
and all positive dilation factors if and only if all observation test
functions vanish.  In particular, any nonzero admissible observation
function breaks the global identity.
Under additional finite Fourier-order assumptions,
Theorem~\ref{thm:quantitative-recovery} identifies the leading
cross-covariance coefficients and the corresponding posterior recovery
rates, including possible cancellations and unequal orders among
observation functions.

The main contribution is uniform two-endpoint recovery on compact target
sets together with this exact symmetry criterion, proved on an admissible
test-function space, with coefficient-level recovery rates under finite
Fourier-order assumptions.  Prior wavelet covariance formulas,
Gaussian projection formulas, and Taylor/moment expansions supply the
underlying tools; the results here determine the normalized posterior
correction and its global symmetry under fixed observations.
As an extension to ordinary processes, we treat the one-dimensional
Mat\'ern-type spectrum
$S(\omega)=C(\kappa^2+\omega^2)^{-p}$ and fixed noisy point data.
The Mat\'ern spectral form and related local scaling results are known
\cite[Eq.~(4.15)]{rasmussen2006} and \cite{limteo2009}; neither the spectral
substitution nor the Gaussian conditioning formula alone is claimed as
new.  The additional result is the compact-uniform posterior recovery
of two \emph{different} prior wavelet covariance kernels at the two
dilation endpoints, with a direct weak-limit proof for the point-data
cross-covariances (Theorem~\ref{thm:tempered-point-endpoints}).
As a separate extension, Proposition~\ref{prop:closed-observation-space}
determines exactly when normalized covariance recovers at either
endpoint after conditioning on a fixed, possibly infinite-dimensional
Gaussian observation subspace.  Symmetric finite-band and high-frequency
subspaces give, respectively, two-sided recovery and a one-sided failure.
Equations~\eqref{eq:mexican-covariance}--\eqref{eq:mexican-limits}
give a closed-form Mexican-hat illustration of the finite-observation
results for $\gamma<5$.
The logarithmic-scale interpretation uses the Lamperti correspondence
\cite{lamperti1962,burnecki1997}.  Neither the dilation identity nor the
posterior limits imply independence between scales or define a
renormalization-group transformation.  A wavelet construction with
explicit coarse-graining, such as that of Marchand and collaborators
\cite{marchand2022}, is a separate operation.  Section~\ref{sec:discussion}
discusses these distinctions and the limits of the fixed-observation
result.

\section{Scope and conventions}
Let $x(t)$ be a real Gaussian process with mean $m(t)$ and covariance
$K(t,t')$.  Let $\psi$ be a real wavelet, and restrict scales to $a>0$.
The $L^2$-normalized continuous wavelet transform is
\begin{equation}
 \psi_{a,b}(t)=a^{-1/2}\psi\!\left(\frac{t-b}{a}\right),
 \qquad
 W(a,b)=\int_{\mathbb R}x(t)\psi_{a,b}(t)\,dt .
 \label{eq:cwt}
\end{equation}
For an ordinary process, the integral is interpreted in mean square:
the centered integral must exist in mean square and the deterministic
mean integral in Eq.~\eqref{eq:mean} must exist and be finite.
For a generalized or intrinsic Gaussian field, $W(a,b)$ means
the corresponding Gaussian linear functional; its test wavelet must be
admissible for that field.  We use the Fourier convention
\begin{equation}
 \widehat f(\omega)=\int_{\mathbb R}f(t)e^{-i\omega t}\,dt,
 \qquad
 \widehat\psi_{a,b}(\omega)
   =\sqrt a\,e^{-i\omega b}\widehat\psi(a\omega).
 \label{eq:fourier}
\end{equation}
All covariance integrals below are assumed finite.  These qualifications
are needed even though the Gaussianity statement itself is elementary.

\section{One coefficient and the full covariance}
For a general Gaussian process, the coefficient mean and covariance are
\begin{align}
 \mu(a,b)&=\int_{\mathbb R}m(t)\psi_{a,b}(t)\,dt,\label{eq:mean}\\
 C_W\bigl((a,b),(a',b')\bigr)
   &=\int_{\mathbb R}\!\int_{\mathbb R}
       \psi_{a,b}(t)K(t,t')\psi_{a',b'}(t')\,dt\,dt'.
       \label{eq:kernel}
\end{align}
In particular, if $V(a,b)=C_W((a,b),(a,b))>0$, then
\begin{equation}
 p_{a,b}(w)=\frac{1}{\sqrt{2\pi V(a,b)}}
 \exp\!\left[-\frac{(w-\mu(a,b))^2}{2V(a,b)}\right].
 \label{eq:single}
\end{equation}
If $V(a,b)=0$, the law is a point mass instead of a density.
For any finite set $z_i=(a_i,b_i)$, the vector
$\bm W=(W(z_1),\ldots,W(z_n))^{\mathsf T}$ is Gaussian with mean
$(\mu(z_i))_i$ and covariance matrix
$\Sigma_{ij}=C_W(z_i,z_j)$.  A multivariate density exists only when
$\Sigma$ is positive definite; continuous-wavelet coefficients may be
linearly dependent for a chosen finite set.  Thus a one-coefficient
variance does not characterize cross-scale dependence.
For a noncentered Gaussian process, the covariance must also be
supplemented by the mean in Eq.~\eqref{eq:mean} to determine these
finite-dimensional laws; the covariance alone does not determine the mean.

\section{Stationary spectral formula and variance scaling}
Now suppose that $x$ is centered and second-order stationary, with an
absolutely continuous spectral measure and
\begin{equation}
 K(t,t')=k(t-t')=\frac{1}{2\pi}
 \int_{\mathbb R}e^{i\omega(t-t')}S(\omega)\,d\omega,
 \qquad S(\omega)\geq 0 .
 \label{eq:psd}
\end{equation}
For an ordinary finite-variance process,
$k(0)=(2\pi)^{-1}\int S(\omega)\,d\omega<\infty$.
Equations~\eqref{eq:fourier} and \eqref{eq:psd} give the standard
spectral-filter formulas \cite{wornell1992,moulines2007}
\begin{align}
 C_W\bigl((a,b),(a',b')\bigr)
 &=\frac{\sqrt{aa'}}{2\pi}\int_{\mathbb R}
  S(\omega)e^{i\omega(b-b')}
  \overline{\widehat\psi(a\omega)}
  \widehat\psi(a'\omega)\,d\omega,
  \label{eq:crossspectral}\\
 V(a)&=\frac{a}{2\pi}\int_{\mathbb R}
  S(\omega)|\widehat\psi(a\omega)|^2\,d\omega.
  \label{eq:variance}
\end{align}
The factor $a$ in Eq.~\eqref{eq:variance} is required by the
$a^{-1/2}$ normalization in Eq.~\eqref{eq:cwt}.
Stationarity makes $V$ independent
of $b$; without stationarity, Eq.~\eqref{eq:kernel} is the appropriate
starting point and $V$ may depend on $b$.
The single-frequency stationary PSD representation in Eq.~\eqref{eq:psd}
is the one used here; no claim is made that nonstationary processes lack
other spectral representations.

Formally set $S(\omega)=C|\omega|^{-\gamma}$, $C>0$, and assume the
wavelet coefficient is well-defined and
\begin{equation}
 A_{\psi,\gamma}:=\frac{1}{2\pi}\int_{\mathbb R}
 |u|^{-\gamma}|\widehat\psi(u)|^2\,du<\infty.
 \label{eq:integrability}
\end{equation}
Changing variables $u=a\omega$ in Eq.~\eqref{eq:variance} yields
\begin{equation}
 \boxed{\quad V(a)=C A_{\psi,\gamma}a^{\gamma}.\quad}
 \label{eq:correctedpower}
\end{equation}
An $a^{-1}$ normalization in Eq.~\eqref{eq:cwt}, with all other
conventions unchanged, would instead give an $a^{\gamma-1}$ variance law.
As a check, for white noise ($\gamma=0$) the
$L^2$-normalized coefficient variance is scale independent.

\paragraph{The status of a pure power law.}
No $\gamma$ makes $C|\omega|^{-\gamma}$ integrable both at $0$
and at infinity: integrability at $0$ requires $\gamma<1$,
whereas integrability at infinity requires $\gamma>1$.
Consequently, Eq.~\eqref{eq:correctedpower} is not a statement about
an ordinary finite-variance stationary process with a pure power law
at every frequency.  It applies, for example, to a suitable generalized
or intrinsic Gaussian model tested against zero-moment wavelets;
alternatively, a regularized stationary spectrum can have a power-law
scaling regime over a restricted range of scales.  The latter does not
imply Eq.~\eqref{eq:correctedpower} exactly at all scales.
If $\widehat\psi(u)=O(|u|^M)$ near $0$ and
$\widehat\psi(u)=O(|u|^{-\beta})$ at infinity, sufficient conditions
for Eq.~\eqref{eq:integrability} are
$\gamma<2M+1$ and $\gamma+2\beta>1$.
Existence of the underlying generalized field, or a regularization of
the ordinary process, must be specified separately.

\section{Cross-scale scaling and the meaning of a flow}
\label{sec:scaling}
Under the same homogeneous-spectrum assumptions, put
$r=a'/a$ and $d=(b'-b)/a$.  Equation~\eqref{eq:crossspectral}
becomes
\begin{equation}
 C_W\bigl((a,b),(a',b')\bigr)
 =C a^{\gamma}F_{\psi,\gamma}(r,d),
 \quad
 F_{\psi,\gamma}(r,d)=\frac{\sqrt r}{2\pi}
 \int_{\mathbb R}|u|^{-\gamma}e^{-iud}
 \overline{\widehat\psi(u)}\widehat\psi(ru)\,du.
 \label{eq:scaling}
\end{equation}
For any $\lambda>0$ this implies the covariance identity
\begin{equation}
 C_W\bigl((\lambda a,\lambda b),(\lambda a',\lambda b')\bigr)
 =\lambda^{\gamma}C_W\bigl((a,b),(a',b')\bigr).
 \label{eq:jointscaling}
\end{equation}
For a centered Gaussian field, Eq.~\eqref{eq:jointscaling} also
gives, for any finite set $(a_i,b_i)$,
\begin{equation*}
 \bigl(\lambda^{-\gamma/2}W(\lambda a_i,\lambda b_i)\bigr)_i
 \stackrel{\mathrm d}{=}
 \bigl(W(a_i,b_i)\bigr)_i.
\end{equation*}
The equality concerns simultaneous scaling of both positions and
scales; it does not assert independence between different scales.

One possible definition of a normalized covariance flow can be made
precise.  For a reference exponent $\gamma_0$, define
\begin{equation*}
 Y_s(r,d)=e^{-\gamma_0s/2}W(e^s r,e^s d),\qquad
 G_s(z,z')=\operatorname{Cov}(Y_s(z),Y_s(z')),
 \quad z=(r,d),\ z'=(r',d').
\end{equation*}
For a general stationary spectrum,
\begin{equation}
 G_s=\frac{\sqrt{rr'}}{2\pi}
 \int_{\mathbb R}e^{-\gamma_0s}S(e^{-s}u)
 e^{iu(d-d')}
 \overline{\widehat\psi(ru)}\widehat\psi(r'u)\,du.
 \label{eq:kernel-flow}
\end{equation}
When differentiating under the integral is justified, its derivative
is obtained by replacing the factor
$e^{-\gamma_0s}S(e^{-s}u)$ in Eq.~\eqref{eq:kernel-flow} with
\begin{equation}
 e^{-\gamma_0s}\left[-\gamma_0 S(e^{-s}u)
             -e^{-s}u S'(e^{-s}u)\right].
 \label{eq:flowderivative}
\end{equation}
For $S(\omega)=C|\omega|^{-\gamma_0}$, $G_s$ is independent of $s$.
This is a dilation fixed-point identity of covariance functions.
Calling it an RG transformation in the statistical-mechanical sense
requires an additional, specified operation that removes or integrates
out degrees of freedom and composes across scales; a variance power law
or Eq.~\eqref{eq:flowderivative} alone is insufficient.

\section{Finite-data posterior in wavelet space}
\label{sec:posterior}
The question studied here is the effect of a fixed finite set of linear
observations on the homogeneous covariance in Eq.~\eqref{eq:jointscaling}.
We first specify the field on which both wavelets and observations act.
The observations remain fixed when the target wavelets are dilated.
Proposition~\ref{prop:posterior-recovery} establishes uniform recovery
of the prior covariance on compact target sets at both endpoints after
posterior normalization.  Theorem~\ref{thm:quantitative-recovery}
quantifies that recovery when finite Fourier orders exist, and
Corollary~\ref{cor:broken-symmetry}
characterizes exactly when the
posterior retains the global prior dilation identity.
We separately test whether the finite observation count is essential
by conditioning on a fixed closed Gaussian observation subspace.

\paragraph{A well-defined homogeneous Gaussian field.}
Fix an integer $M\geq1$, $C>0$, and a real exponent $\gamma<2M+1$.
Let
\begin{equation*}
 \mathcal S_M=\left\{\phi\in\mathcal S(\mathbb R;\mathbb R):
 \int_{\mathbb R}t^k\phi(t)\,dt=0,\quad k=0,\ldots,M-1\right\},
\end{equation*}
where $\mathcal S$ is the Schwartz space, and define
\begin{equation*}
 Q(\phi,\chi)=\frac{C}{2\pi}\int_{\mathbb R}|\omega|^{-\gamma}
       \overline{\widehat\phi(\omega)}\widehat\chi(\omega)\,d\omega.
\end{equation*}
Vanishing moments imply $\widehat\phi(\omega)=O(|\omega|^M)$ at zero.
Thus the integrand defining $Q(\phi,\phi)$ is bounded near zero by
a constant times $|\omega|^{2M-\gamma}$, which is integrable.
Schwartz decay guarantees integrability at infinity for every such
$\gamma$; mixed integrals exist by Cauchy--Schwarz.
Hermitian symmetry of Fourier transforms of real functions makes $Q$
real and symmetric.  Moreover, $Q(\phi,\phi)>0$ for every nonzero
$\phi\in\mathcal S_M$, because the spectral weight is strictly positive
almost everywhere.  All finite Gram matrices are consequently positive
semidefinite, and positive definite for linearly independent test
functions.  The same estimates show continuity in the Schwartz topology.

For an explicit Gaussian construction, complete $\mathcal S_M$ in the
norm $Q(\phi,\phi)^{1/2}$ to a real separable Hilbert space $\mathcal H_Q$.
Separability follows from its isometric embedding into $L^2$ given below.
Choose an orthonormal basis $(e_k)$ and independent standard real
Gaussian variables $(\xi_k)$, and set
\begin{equation*}
 X(\phi)=\sum_k\langle\phi,e_k\rangle_{\mathcal H_Q}\xi_k,
 \qquad \phi\in\mathcal S_M.
\end{equation*}
The series converges in mean square.  This defines a centered, mean-square
linear Gaussian family with covariance $Q$; no pointwise random function
is assumed.  Fix a nonzero real wavelet $\psi\in\mathcal S_M$ and set
$W(a,b)=X(\psi_{a,b})$ with the normalization in Eq.~\eqref{eq:cwt}.
Translations and the dilations
$T_\lambda\phi(t)=\lambda^{-1/2}\phi(t/\lambda)$ preserve $\mathcal S_M$.
Indeed, a moment of a translated, dilated test function is a linear
combination of its moments of order at most $k<M$.
In particular, $T_\lambda\psi_{a,b}=\psi_{\lambda a,\lambda b}$.
Substitution in $Q$ gives Eqs.~\eqref{eq:crossspectral}--\eqref{eq:jointscaling}
with $S(\omega)=C|\omega|^{-\gamma}$, now as well-defined field covariances.
Point evaluations are not part of this construction: the formal variance
of $X(\delta_t)$ would require $\int|\omega|^{-\gamma}\,d\omega$, which
diverges for every $\gamma$.

\paragraph{Fixed finite observations and Gaussian conditioning.}
Let $n\geq1$ be finite, fix $\phi_1,\ldots,\phi_n\in\mathcal S_M$, and take
$y_i=X(\phi_i)+\varepsilon_i$.
The centered Gaussian vector $\bm\varepsilon$ is independent of the
entire field and has a fixed positive-semidefinite covariance matrix $N$.
Assume
\begin{equation*}
 \begin{aligned}
 A=\operatorname{Cov}(\bm y)
   &=\bigl(Q(\phi_i,\phi_j)\bigr)_{ij}+N>0,\\
 q_z&=\bigl(Q(\psi_{a,b},\phi_i)\bigr)_{i=1}^n,
 \quad z=(a,b).
 \end{aligned}
\end{equation*}
Here $A>0$ means positive definite; noiseless observations are allowed
when their Gram matrix is invertible.  Neither $n$, the $\phi_i$, nor $N$
depends on the dilation parameter.  Finite-dimensional Gaussian
conditioning gives
\begin{equation}
 C_W^{\mathrm{post}}(z,z')
   =C_W(z,z')-q_z^{\mathsf T}A^{-1}q_{z'}.
 \label{eq:posterior}
\end{equation}
For example, $W(z)-q_z^{\mathsf T}A^{-1}\bm y$ has zero covariance
with $\bm y$, hence is independent of $\bm y$ by joint Gaussianity;
its covariance is Eq.~\eqref{eq:posterior}.
The positive-semidefinite correction kernel
$\Delta(z,z')=q_z^{\mathsf T}A^{-1}q_{z'}$ has rank at most $n$.
These conditioning and finite-rank statements are established facts,
not new results \cite{rasmussen2006,sottinen2014}.
Independence of non-Gaussian noise alone would not justify this
conditional covariance formula.

\paragraph{Normalized dilations.}
Use the usual complex $L^2(\mathbb R,d\omega)$ inner product
$\langle f,g\rangle_2=\int\overline f g\,d\omega$ and define
\begin{equation}
 J\phi(\omega)=\sqrt{\frac{C}{2\pi}}\,
       |\omega|^{-\gamma/2}\widehat\phi(\omega),
 \qquad D_\lambda h(\omega)=\sqrt\lambda\,h(\lambda\omega).
 \label{eq:spectral-isometry}
\end{equation}
The value at $\omega=0$ is immaterial.
The image of real tests has Hermitian symmetry and is regarded as a
real Hilbert subspace.  On this image,
$Q(\phi,\chi)=\langle J\phi,J\chi\rangle_2$ is real.
The maps $D_\lambda$ are unitary, with inverse $D_{1/\lambda}$, and
\begin{equation}
 JT_\lambda\phi=\lambda^{\gamma/2}D_\lambda J\phi.
 \label{eq:dilation-intertwining}
\end{equation}
Thus the power $\gamma/2$ follows from the original $L^2$ normalization.

\begin{lemma}[Weak convergence of dilations]
\label{lem:weak-dilation}
For every $h\in L^2(\mathbb R)$,
$D_\lambda h\rightharpoonup0$ both as $\lambda\downarrow0$ and as
$\lambda\uparrow\infty$.
\end{lemma}
\noindent\textit{Proof.}
For arbitrary $h,g\in L^2$, truncate both to the annulus
$E_{\delta,R}=\{\omega:\delta\leq|\omega|\leq R\}$, where
$0<\delta<R$, obtaining $h_E,g_E$.
The supports of $D_\lambda h_E$ and $g_E$ are disjoint whenever
$\lambda<\delta/R$ or $\lambda>R/\delta$, so their inner product is zero.
For all $\lambda>0$, unitarity and Cauchy--Schwarz give
\begin{equation*}
 |\langle D_\lambda h,g\rangle_2
       -\langle D_\lambda h_E,g_E\rangle_2|
 \leq\|h-h_E\|_2\|g\|_2+\|h_E\|_2\|g-g_E\|_2.
\end{equation*}
The right side tends to zero as $\delta\downarrow0$, $R\uparrow\infty$,
independently of $\lambda$.  Taking either dilation limit first proves
the assertion.\hfill$\square$

\begin{proposition}[Recovery at both dilation limits]
\label{prop:posterior-recovery}
For the homogeneous Gaussian field and fixed finite observations defined
above, write $\lambda z=(\lambda a,\lambda b)$.
For every fixed $z=(a,b)$ and $z'=(a',b')$ with $a,a'>0$,
\begin{equation}
 \begin{aligned}
 \lim_{\lambda\downarrow0}\lambda^{-\gamma}
       C_W^{\mathrm{post}}(\lambda z,\lambda z')&=C_W(z,z'),\\
 \lim_{\lambda\uparrow\infty}\lambda^{-\gamma}
       C_W^{\mathrm{post}}(\lambda z,\lambda z')
 &=C_W(z,z').
 \end{aligned}
 \label{eq:posterior-recovery}
\end{equation}
In particular, the convergence is pointwise in the pair $(z,z')$.
More strongly, for every nonempty compact
$K\subset(0,\infty)\times\mathbb R$,
\begin{equation*}
 E_\lambda(K):=\sup_{z,z'\in K}
 \left|\lambda^{-\gamma}C_W^{\mathrm{post}}(\lambda z,\lambda z')
       -C_W(z,z')\right|,
\end{equation*}
and
\begin{equation*}
 \lim_{\lambda\downarrow0}E_\lambda(K)
 =\lim_{\lambda\uparrow\infty}E_\lambda(K)=0.
\end{equation*}
\end{proposition}
\noindent\textit{Proof.}
Put $h_z=J\psi_{a,b}$ and $g_i=J\phi_i$.
Equation~\eqref{eq:dilation-intertwining} gives
\begin{equation*}
 v_\lambda(z):=\lambda^{-\gamma/2}q_{\lambda z},
 \qquad [v_\lambda(z)]_i=\langle D_\lambda h_z,g_i\rangle_2.
\end{equation*}
Each entry tends to zero in both limits by Lemma~\ref{lem:weak-dilation}.
Since $n$ is fixed and finite, $v_\lambda(z)\to0$ in $\mathbb R^n$.
Combining Eq.~\eqref{eq:posterior} with Eq.~\eqref{eq:jointscaling} yields
\begin{equation}
 \lambda^{-\gamma}C_W^{\mathrm{post}}(\lambda z,\lambda z')
 =C_W(z,z')-v_\lambda(z)^{\mathsf T}A^{-1}v_\lambda(z').
 \label{eq:normalized-correction}
\end{equation}
The absolute value of the last term is bounded by
$\|A^{-1}\|\,\|v_\lambda(z)\|\,\|v_\lambda(z')\|$, which tends to zero
at either endpoint.
To obtain uniformity on a fixed compact $K$, note that
$h_{a,b}(\omega)=a^{\gamma/2}e^{-i\omega b}D_a(J\psi)(\omega)$.
Dilation and multiplication by $e^{-i\omega b}$ are strongly continuous
on $L^2$ as $a>0$ and $b\in\mathbb R$ vary (first on smooth compactly
supported functions, then by density and unitarity).
Thus $\{h_z:z\in K\}$ is compact in $L^2$.
For each $i$, let $h_1,\ldots,h_m$ be a finite $\epsilon$-net in this set.
Lemma~\ref{lem:weak-dilation} and $\|D_\lambda\|=1$ give
\begin{equation*}
 \sup_{z\in K}|[v_\lambda(z)]_i|
 \leq\max_{1\leq j\leq m}|\langle D_\lambda h_j,g_i\rangle_2|
       +\epsilon\|g_i\|_2.
\end{equation*}
The maximum tends to zero at both endpoints, and $\epsilon>0$ is
arbitrary.
Since $n$ is finite, $\sup_{z\in K}\|v_\lambda(z)\|\to0$ at both endpoints.
Equation~\eqref{eq:normalized-correction} bounds $E_\lambda(K)$\\
by $\|A^{-1}\|\bigl(\sup_{z\in K}\|v_\lambda(z)\|\bigr)^2$, proving the claim.\hfill$\square$

\paragraph{Conditioning on a closed Gaussian observation subspace.}
Let $\mathcal H$ be the closure of $J\mathcal S_M$ in the real Hilbert
space of Hermitian-symmetric functions in $L^2(\mathbb R,d\omega)$.
In fact, $\mathcal H$ is the entire real Hilbert space of such functions.
To see this, approximate any Hermitian-symmetric $L^2$ function by
smooth Hermitian-symmetric functions $f$ supported on symmetric
compact sets away from zero.  For each such $f$,
$\widehat\phi(\omega)=\sqrt{2\pi/C}\,|\omega|^{\gamma/2}f(\omega)$
is the Fourier transform of a real Schwartz function with all moments
zero, so $\phi\in\mathcal S_M$ and $J\phi=f$.
The Gaussian family $X$ therefore extends by mean-square continuity
to an isonormal family $X_{\mathcal H}(h)$, $h\in\mathcal H$, with
$X_{\mathcal H}(J\phi)=X(\phi)$.

Fix a closed real subspace $E\subseteq\mathcal H$ independent of
$\lambda$, and let $\mathcal F_E$ be the completed sigma-field
generated by $\{X_{\mathcal H}(e):e\in E\}$.  Since $\mathcal H$ is
separable, a countable orthonormal basis of $E$ generates the same
completed sigma-field.  Thus conditioning on all of $E$ has a precise
meaning even when $E$ is infinite-dimensional; it does not posit an
$E$-valued Gaussian random vector with identity covariance.
Write $P_E$ for the orthogonal projection onto $E$, and define
$C_W^E(z,z')=\operatorname{Cov}(W(z),W(z')\mid\mathcal F_E)$.
These are exact, noiseless Gaussian-coordinate observations.  The
finite noisy model in Eq.~\eqref{eq:posterior} remains distinct: its
$A^{-1}$ formula is a finite-dimensional projection in the enlarged
Gaussian space containing both the field and the independent noise.

\begin{proposition}[Endpoint criterion for closed observation spaces]
\label{prop:closed-observation-space}
For the homogeneous field above, let $h_z=J\psi_{a,b}$ and fix a
closed observation subspace $E\subseteq\mathcal H$.  Then
\begin{equation}
 C_W^E(z,z')
 =\langle(I-P_E)h_z,(I-P_E)h_{z'}\rangle_2
 =C_W(z,z')-\langle P_Eh_z,P_Eh_{z'}\rangle_2.
 \label{eq:closed-space-posterior}
\end{equation}
At either endpoint $\lambda\downarrow0$ or
$\lambda\uparrow\infty$, recovery
$\lambda^{-\gamma}C_W^E(\lambda z,\lambda z')\to C_W(z,z')$
for every fixed target pair $z,z'$ holds if and only if
\begin{equation}
 \|P_ED_\lambda h_z\|_2\longrightarrow0
 \quad\text{for every fixed }z=(a,b),\ a>0.
 \label{eq:closed-space-criterion}
\end{equation}
For every nonempty compact $K\subset(0,\infty)\times\mathbb R$,
the exact uniform-error identity is
\begin{equation}
 \sup_{z,z'\in K}\left|
 \lambda^{-\gamma}C_W^E(\lambda z,\lambda z')-C_W(z,z')
 \right|
 =\left(\sup_{z\in K}\|P_ED_\lambda h_z\|_2\right)^2.
 \label{eq:closed-space-uniform-error}
\end{equation}
Hence compact-uniform recovery at either endpoint is equivalent to
$\sup_{z\in K}\|P_ED_\lambda h_z\|_2\to0$ for every such $K$.
In the present wavelet family, pointwise condition
\eqref{eq:closed-space-criterion} already implies this compact-uniform
condition at the same endpoint.
\end{proposition}
\noindent\textit{Proof.}
Orthogonality makes
$X_{\mathcal H}((I-P_E)h_z)$ independent of every finite collection of
Gaussian coordinates in $E$, hence independent of $\mathcal F_E$.
Thus the conditional mean is $X_{\mathcal H}(P_Eh_z)$ and the
conditional covariance is the deterministic kernel
\eqref{eq:closed-space-posterior}.
Using Eq.~\eqref{eq:dilation-intertwining},
\begin{equation*}
 \lambda^{-\gamma}C_W^E(\lambda z,\lambda z')
 =C_W(z,z')-
 \langle P_ED_\lambda h_z,P_ED_\lambda h_{z'}\rangle_2.
\end{equation*}
The diagonal correction is $\|P_ED_\lambda h_z\|_2^2$;
this proves necessity of Eq.~\eqref{eq:closed-space-criterion}.
Cauchy--Schwarz proves sufficiency and, using the diagonal again,
gives the equality in Eq.~\eqref{eq:closed-space-uniform-error}.
Finally, $\{h_z:z\in K\}$ is compact by the continuity proved in
Proposition~\ref{prop:posterior-recovery}, and
$\|P_ED_\lambda\|\leq1$.  A finite $\epsilon$-net therefore
upgrades pointwise convergence of $P_ED_\lambda h_z$ to uniform
convergence for $z\in K$.\hfill$\square$

\paragraph{Infinite-dimensional frequency observations.}
For a symmetric measurable frequency set $B$, let $E_B$ consist of
all $h\in\mathcal H$ that vanish almost everywhere outside $B$.
Then $E_B$ is closed and $P_{E_B}h=\mathbf1_Bh$; the following two
subspaces are both infinite-dimensional.  For the finite band
$B_{\mathrm{band}}=\{\omega:1\leq|\omega|\leq2\}$, substitution
$u=\lambda\omega$ gives
\begin{equation*}
 \|P_{E_{B_{\mathrm{band}}}}D_\lambda h_z\|_2^2
 =\int_{\lambda\leq|u|\leq2\lambda}|h_z(u)|^2\,du
 \longrightarrow0
 \quad(\lambda\downarrow0\ \text{or}\ \lambda\uparrow\infty).
\end{equation*}
Indeed, the integration annulus contracts to zero or escapes to
infinity, and $|h_z|^2$ is integrable.  Proposition~
\ref{prop:closed-observation-space} therefore gives recovery at both
endpoints, uniformly on every fixed compact target set.

In contrast, for $B_{\mathrm{high}}=\{\omega:|\omega|\geq1\}$,
\begin{equation*}
 \|P_{E_{B_{\mathrm{high}}}}D_\lambda h_z\|_2^2
 =\int_{|u|\geq\lambda}|h_z(u)|^2\,du
 \longrightarrow
 \begin{cases}
 \|h_z\|_2^2=C_W(z,z)>0,&\lambda\downarrow0,\\
 0,&\lambda\uparrow\infty.
 \end{cases}
\end{equation*}
At the small-scale endpoint, the unobserved component has squared
norm $\int_{|u|<\lambda}|h_z(u)|^2\,du\to0$.  Applying
Cauchy--Schwarz to Eq.~\eqref{eq:closed-space-posterior} yields
\begin{equation*}
 \lim_{\lambda\downarrow0}\lambda^{-\gamma}
 C_W^{E_{B_{\mathrm{high}}}}(\lambda z,\lambda z')=0,
 \qquad
 \lim_{\lambda\uparrow\infty}\lambda^{-\gamma}
 C_W^{E_{B_{\mathrm{high}}}}(\lambda z,\lambda z')=C_W(z,z').
\end{equation*}
Both limits are uniform on fixed compact target sets by the same
finite-net argument.  Since the prior variance is strictly positive,
small-scale recovery fails.  Thus finite observation count is
sufficient for two-endpoint recovery but is not necessary; a fixed
closed infinite-dimensional observation space can preserve both
endpoints or eliminate one of them.

The quantitative result below uses the Fourier counterpart of the
established moment-expansion method.  Pathak and Pathak
\cite[Theorems~3.1 and~3.3]{pathak2014} give large- and small-scale
wavelet-transform expansions at fixed translation, respectively.
Here the homogeneous spectral weight and simultaneous target dilation
enter the cross-covariance coefficients, which are then propagated
through the quadratic posterior correction in
Eq.~\eqref{eq:normalized-correction}.  The resulting coefficients,
recovery orders, and cancellation bounds hold uniformly on every fixed
compact target set; the Taylor expansion itself is not claimed as new.

\begin{theorem}[Quantitative two-endpoint recovery]
\label{thm:quantitative-recovery}
In addition to the assumptions of Proposition~\ref{prop:posterior-recovery},
suppose $\widehat\psi$ first has a nonzero Taylor coefficient at the
finite order $m_\psi$, and every nonzero $\phi_i$ has one at the finite
order $m_i$.  Let $I=\{i:\phi_i\ne0\}$; if $I$ is empty, the posterior
correction is identically zero.  Otherwise put
\begin{equation*}
 \begin{aligned}
 c_\psi&=\frac{\widehat\psi^{(m_\psi)}(0)}{m_\psi!},\qquad
 &c_i&=\frac{\widehat\phi_i^{(m_i)}(0)}{m_i!},\\
 \alpha&=m_\psi+\frac{1-\gamma}{2},
 &\beta_i&=m_i+\frac{1-\gamma}{2},\\
 m_*&=\min_{i\in I}m_i,
 &\beta_*&=m_*+\frac{1-\gamma}{2}.
 \end{aligned}
\end{equation*}
For $i\in I$ and $z=(a,b)$, define the real coefficients
\begin{align}
 S_i(z)&=\frac{C a^{m_\psi+1/2}\overline{c_\psi}}{2\pi}
  \int_{\mathbb R}|\omega|^{-\gamma}\omega^{m_\psi}
      \widehat\phi_i(\omega)\,d\omega,\label{eq:small-cross-coefficient}\\
 L_i(z)&=\frac{C\sqrt a\,c_i}{2\pi}
  \int_{\mathbb R}|u|^{-\gamma}u^{m_i}e^{ibu}
      \overline{\widehat\psi(au)}\,du.
      \label{eq:large-cross-coefficient}
\end{align}
Set $S_i=L_i=0$ for $i\notin I$, write $S(z)=(S_i(z))_i$, and let
$L_*(z)$ have entry $L_i(z)$ if $i\in I$ and $m_i=m_*$, and zero
otherwise.  Then for every fixed nonempty compact
$K\subset(0,\infty)\times\mathbb R$, uniformly in $z\in K$,
\begin{equation}
 [v_\lambda(z)]_i=
 \begin{cases}
 \lambda^\alpha\bigl(S_i(z)+O_K(\lambda)\bigr),
       &\lambda\downarrow0,\quad i\in I,\\
 \lambda^{-\beta_i}\bigl(L_i(z)+O_K(\lambda^{-1})\bigr),
       &\lambda\uparrow\infty,\quad i\in I.
 \end{cases}
 \label{eq:cross-asymptotics}
\end{equation}
The entries with $i\notin I$ are exactly zero.  Consequently, uniformly
for $z,z'\in K$,
\begin{align}
 \lambda^{-\gamma}C_W^{\mathrm{post}}(\lambda z,\lambda z')-C_W(z,z')
 &=-\lambda^{2\alpha}S(z)^{\mathsf T}A^{-1}S(z')
     +O_K(\lambda^{2\alpha+1}),&&\lambda\downarrow0,
     \label{eq:small-posterior-rate}\\
 &=-\lambda^{-2\beta_*}L_*(z)^{\mathsf T}A^{-1}L_*(z')
     +O_K(\lambda^{-2\beta_*-1}),&&\lambda\uparrow\infty.
     \label{eq:large-posterior-rate}
\end{align}
In particular, for the $E_\lambda(K)$ of
Proposition~\ref{prop:posterior-recovery}, set
\begin{equation*}
 B_0(K)=\max_{z\in K}S(z)^{\mathsf T}A^{-1}S(z),\qquad
 B_\infty(K)=\max_{z\in K}L_*(z)^{\mathsf T}A^{-1}L_*(z).
\end{equation*}
Then
\begin{equation}
 \begin{aligned}
 E_\lambda(K)&=\lambda^{2m_\psi+1-\gamma}
      \bigl(B_0(K)+O_K(\lambda)\bigr)\quad(\lambda\downarrow0),\\
 E_\lambda(K)&=\lambda^{-(2m_*+1-\gamma)}
      \bigl(B_\infty(K)+O_K(\lambda^{-1})\bigr)
      \quad(\lambda\uparrow\infty).
 \end{aligned}
 \label{eq:uniform-posterior-rate}
\end{equation}
Both exponents are positive.  If the relevant $B$ is positive, the
displayed exponent is the exact leading recovery exponent on $K$.
If $B_0(K)=0$, the first bound improves to
$E_\lambda(K)=O_K(\lambda^{2\alpha+2})$; if $B_\infty(K)=0$, the second
improves to $E_\lambda(K)=O_K(\lambda^{-2\beta_*-2})$.
\end{theorem}
\noindent\textit{Proof.}
The definition of $Q$ and the Fourier transform in Eq.~\eqref{eq:fourier}
give two equivalent representations:
\begin{equation*}
 \begin{aligned}
 [v_\lambda(z)]_i
 &=\frac{C\sqrt a}{2\pi}\lambda^{(1-\gamma)/2}
   \int_{\mathbb R}|\omega|^{-\gamma}e^{i\lambda b\omega}
   \overline{\widehat\psi(\lambda a\omega)}
   \widehat\phi_i(\omega)\,d\omega\\
 &=\frac{C\sqrt a}{2\pi}\lambda^{(\gamma-1)/2}
   \int_{\mathbb R}|u|^{-\gamma}e^{ibu}
   \overline{\widehat\psi(au)}\widehat\phi_i(u/\lambda)\,du.
 \end{aligned}
\end{equation*}
Since $m_\psi,m_i\geq M$ and $\gamma<2M+1$,
$m_\psi+m_i-\gamma>-1$; the displayed coefficient integrals
therefore converge at zero, and Schwartz decay controls infinity.
They are real by Hermitian symmetry.  Taylor expansion of
$e^{i\lambda b\omega}\overline{\widehat\psi(\lambda a\omega)}$
in the first representation and of $\widehat\phi_i(u/\lambda)$
in the second gives Eq.~\eqref{eq:cross-asymptotics}.
For $z$ in a fixed compact $K$, Taylor remainders are bounded by
constants times $\lambda^{m_\psi+1}|\omega|^{m_\psi+1}$ and
$\lambda^{-m_i-1}|u|^{m_i+1}$, respectively.  The resulting weighted
integrals are finite, uniformly in $z\in K$, so both stated error
bounds are uniform.  Substituting these expansions into
Eq.~\eqref{eq:normalized-correction} proves
Eqs.~\eqref{eq:small-posterior-rate} and
\eqref{eq:large-posterior-rate}; entries with $m_i>m_*$ have at least
one additional inverse power of $\lambda$.
Positive definiteness of $A^{-1}$ and Cauchy--Schwarz show that
$E_\lambda(K)=\max_{z\in K}v_\lambda(z)^{\mathsf T}A^{-1}v_\lambda(z)$:
the cross-pair supremum is attained on the diagonal.  This identity
and the uniform expansions give Eq.~\eqref{eq:uniform-posterior-rate}.
If $B_0(K)=0$, positive definiteness gives $S(z)=0$ throughout $K$;
thus $v_\lambda(z)=O_K(\lambda^{\alpha+1})$.
If $B_\infty(K)=0$, $L_*(z)=0$ throughout $K$ and every component of
$v_\lambda(z)$ is $O_K(\lambda^{-\beta_*-1})$.
Squaring these bounds proves the two improvements.\hfill$\square$

The theorem therefore quantifies the fixed-observation posterior error,
not just an individual wavelet transform: its exact compact-set error
coefficients are $B_0(K)$ and $B_\infty(K)$.
The leading coefficients may vanish.  For example, opposite parity of
$\psi$ and $\phi_i$ about the origin makes the weighted integral in
Eq.~\eqref{eq:small-cross-coefficient} zero and makes
$L_i(a,0)=0$; weighted orthogonality can also annul the small-scale
integral without parity.  For distinct targets, even nonzero $L_*$
vectors can have $L_*(z)^{\mathsf T}A^{-1}L_*(z')=0$; on the diagonal
the coefficient is positive whenever its vector is nonzero.
Here $S(z)=a^{m_\psi+1/2}h$ for a constant vector $h$, so if $h\ne0$
its cross-pair coefficient is positive for every $z,z'$.  At large
scales $L_*(z)$ can depend on both $a$ and $b$.
If a displayed coefficient vanishes, its candidate power is only an
upper bound; higher Taylor coefficients, where nonzero, determine the
first actual power.  The original endpoint limits remain valid when a
nonzero Schwartz function is Fourier-flat at zero, even though its
finite first nonzero order required in this proposition does not exist.

For clarity, the later candidate coefficients are also explicit.  Write
$c_{\psi,j}=\widehat\psi^{(j)}(0)/j!$ and
$c_{i,r}=\widehat\phi_i^{(r)}(0)/r!$.  For $r\geq m_\psi$ and
$s\geq m_i$, put
\begin{align*}
 S_{i,r}(a,b)&=\frac{C\sqrt a}{2\pi}
   \left(\int_{\mathbb R}|\omega|^{-\gamma}\omega^r
      \widehat\phi_i(\omega)\,d\omega\right)
   \sum_{j=m_\psi}^{r}a^j\overline{c_{\psi,j}}
      \frac{(ib)^{r-j}}{(r-j)!},\\
 L_{i,s}(a,b)&=\frac{C\sqrt a\,c_{i,s}}{2\pi}
   \int_{\mathbb R}|u|^{-\gamma}u^s e^{ibu}
      \overline{\widehat\psi(au)}\,du.
\end{align*}
For any finite $R\geq m_\psi$ and $R'\geq m_i$, the same proof yields,
uniformly on fixed compact target sets,
\begin{align*}
 [v_\lambda(z)]_i
 &=\sum_{r=m_\psi}^{R}
     \lambda^{r+(1-\gamma)/2}S_{i,r}(z)
     +O_K\bigl(\lambda^{R+1+(1-\gamma)/2}\bigr)
       &&(\lambda\downarrow0),\\
 [v_\lambda(z)]_i
 &=\sum_{s=m_i}^{R'}
     \lambda^{-[s+(1-\gamma)/2]}L_{i,s}(z)
     +O_K\bigl(\lambda^{-[R'+1+(1-\gamma)/2]}\bigr)
       &&(\lambda\uparrow\infty).
\end{align*}
Thus a later nonzero coefficient can identify a faster actual rate;
if all coefficients at a fixed target vanish, its cross-covariance
decays faster than every algebraic power at that endpoint.

\begin{corollary}[Exact dilation symmetry]
\label{cor:broken-symmetry}
Under the assumptions of Proposition~\ref{prop:posterior-recovery}, the identity
\begin{equation*}
 C_W^{\mathrm{post}}(\lambda z,\lambda z')
   =\lambda^\gamma C_W^{\mathrm{post}}(z,z')
\end{equation*}
holds for all $\lambda>0$ and all target pairs $z,z'$ if and only if
the observation correction vanishes identically on the wavelet family:
\begin{equation*}
 \Delta(z,z')=q_z^{\mathsf T}A^{-1}q_{z'}=0
 \qquad\text{for all }z,z'.
\end{equation*}
Equivalently, $q_z=0$ for every $z$.
For the nonzero real Schwartz wavelet fixed above, this occurs if and
only if every observation test function $\phi_i$ is zero.
In particular, if any $\phi_i$ is nonzero, the global dilation identity
cannot hold.
\end{corollary}
\noindent\textit{Proof.}
If the identity holds globally, then for each fixed target pair
$\lambda^{-\gamma}C_W^{\mathrm{post}}(\lambda z,\lambda z')$
equals $C_W^{\mathrm{post}}(z,z')$ for every $\lambda$.
Proposition~\ref{prop:posterior-recovery} identifies either endpoint limit
with $C_W(z,z')$, so Eq.~\eqref{eq:posterior} gives $\Delta(z,z')=0$
for every pair.
Conversely, if the correction vanishes identically, the posterior
covariance equals the prior covariance and obeys Eq.~\eqref{eq:jointscaling}.
Finally, positive definiteness of $A^{-1}$ implies that
$\Delta(z,z)=0$ if and only if $q_z=0$.
Thus zero correction on the entire family is equivalent to $q_z=0$
for every $z$, and any positive variance reduction precludes the
global identity.
To show that a nonzero observation function necessarily gives such a
reduction, fix $i$ with $\phi_i\ne0$.
The Fourier transforms of the nonzero real Schwartz functions $\psi$
and $\phi_i$ are continuous and have Hermitian symmetry; hence there
are $p,q>0$ with $\widehat\psi(p)\ne0$ and $\widehat\phi_i(q)\ne0$.
Choose $a=p/q$.  The function
$f_a(\omega)=|\omega|^{-\gamma}
\overline{\widehat\psi(a\omega)}\widehat\phi_i(\omega)$
is nonzero on an interval around $q$, and is in $L^1(\mathbb R)$
by the finiteness of $Q$ and Cauchy--Schwarz.
By Eq.~\eqref{eq:fourier} and the definition of $Q$,
\begin{equation*}
 Q(\psi_{a,b},\phi_i)
 =\frac{C\sqrt a}{2\pi}\int_{\mathbb R}
      e^{i\omega b}f_a(\omega)\,d\omega.
\end{equation*}
Uniqueness of the Fourier transform on $L^1$ implies that this quantity
cannot vanish for every $b$.  For some $z=(a,b)$, therefore, $q_z\ne0$
and $\Delta(z,z)>0$ because $A^{-1}>0$.
Conversely, if every $\phi_i=0$, then $q_z=0$ for every $z$.
This completes the characterization.\hfill$\square$

The corollary concerns a global identity; it does not rule out equality
at particular scale pairs.

\paragraph{Scope of the limits.}
The observations are fixed while both target scales and target positions
are dilated about the origin.  The result also gives entrywise convergence
for any fixed finite collection of targets and uniform convergence when
the targets range over a fixed compact subset of
$(0,\infty)\times\mathbb R$.  It asserts neither uniform convergence
over all targets nor convergence in a covariance-operator norm, and
the rates in Theorem~\ref{thm:quantitative-recovery} require its
additional finite Fourier-order assumptions.
Simultaneously dilating the observation functions or allowing their
number to grow with $\lambda$ are different problems.  Proposition~
\ref{prop:closed-observation-space} treats fixed infinite-dimensional
subspaces through their exact Gaussian coordinates; an unspecified
infinite collection of noisy data is not covered by that statement.
The proofs use the specific homogeneous spectral Hilbert space and do
not extend by assertion to arbitrary Gaussian processes, including
models with additional finite-dimensional self-similar components.

\paragraph{A finite-variance Mat\'ern-type process with point data.}
There is a distinct two-endpoint result for ordinary processes that
avoids the inadmissibility of point evaluations under a pure power
spectrum.  Fix the real nonzero $\psi\in\mathcal S_M$ above, $M\geq1$,
and consider a centered stationary Gaussian process with
\begin{equation}
 S_{p,\kappa}(\omega)=C(\kappa^2+\omega^2)^{-p},\qquad
 C,\kappa>0,\quad \frac12<p<M+\frac12.
 \label{eq:tempered-spectrum}
\end{equation}
In one dimension this is a Mat\'ern spectral form with smoothness
$\nu=p-1/2$ after matching the length and amplitude parameters
\cite[Eq.~(4.15)]{rasmussen2006}.  Here the Fourier variable is angular
frequency, as in Eq.~\eqref{eq:fourier}.  Define the process covariance
$k(s)=(2\pi)^{-1}\int e^{i\omega s}S_{p,\kappa}(\omega)\,d\omega$.
At fixed locations $t_1,\ldots,t_n$, observe
$y_i=x(t_i)+\varepsilon_i$, with centered Gaussian noise independent
of the process, fixed covariance $N\geq0$, and
$A=(k(t_i-t_j))_{ij}+N>0$.
Neither the locations, the noise covariance, nor $n$ changes with
$\lambda$.  Write $q_z=(\operatorname{Cov}(W(z),x(t_i)))_i$.

\begin{theorem}[Distinct posterior endpoints with fixed point data]
\label{thm:tempered-point-endpoints}
\leavevmode\\
Under these assumptions, let $z=(a,b)$ and $z'=(a',b')$, with
$a,a'>0$.  Define the two prior wavelet kernels
\begin{align}
 H_{2p}(z,z')&=\frac{C\sqrt{aa'}}{2\pi}
  \int_{\mathbb R}|u|^{-2p}e^{iu(b-b')}
  \overline{\widehat\psi(au)}\widehat\psi(a'u)\,du,
  \label{eq:tempered-high-limit}\\
 H_0(z,z')&=S_{p,\kappa}(0)
   \int_{\mathbb R}\psi_{a,b}(t)\psi_{a',b'}(t)\,dt,
   \qquad S_{p,\kappa}(0)=C\kappa^{-2p}.
   \label{eq:tempered-low-limit}
\end{align}
For every nonempty compact $K\subset(0,\infty)\times\mathbb R$,
\begin{equation}
 \begin{aligned}
 \lim_{\lambda\downarrow0}\sup_{z,z'\in K}
  \left|\lambda^{-2p}C_W^{\mathrm{post}}(\lambda z,\lambda z')
       -H_{2p}(z,z')\right|&=0,\\
 \lim_{\lambda\uparrow\infty}\sup_{z,z'\in K}
  \left|C_W^{\mathrm{post}}(\lambda z,\lambda z')
       -H_0(z,z')\right|&=0.
 \end{aligned}
 \label{eq:tempered-posterior-limits}
\end{equation}
In particular, for each $i$ the cross-covariances satisfy
\begin{equation}
 \sup_{z\in K}\left|\lambda^{-p}
       \operatorname{Cov}(W(\lambda z),x(t_i))\right|\to0
       \quad(\lambda\downarrow0),\qquad
 \sup_{z\in K}\left|\operatorname{Cov}(W(\lambda z),x(t_i))\right|
       \to0\quad(\lambda\uparrow\infty).
 \label{eq:tempered-point-cross-limits}
\end{equation}
\end{theorem}
\noindent\textit{Proof.}
Changing variables $u=\lambda\omega$ in
Eq.~\eqref{eq:crossspectral} yields a prior covariance with spectral
factor $S_{p,\kappa}(u/\lambda)$.
At the small-scale endpoint,
$\lambda^{-2p}S_{p,\kappa}(u/\lambda)
=C(\kappa^2\lambda^2+u^2)^{-p}\to C|u|^{-2p}$ for $u\ne0$;
the factor is bounded by $C|u|^{-2p}$.  At the large-scale endpoint,
$S_{p,\kappa}(u/\lambda)\to S_{p,\kappa}(0)$ and is bounded by
$S_{p,\kappa}(0)$.  On $z,z'\in K$, the Fourier transforms of the
Schwartz wavelets have uniformly bounded envelopes of order $|u|^M$
at zero and faster-than-polynomial decay at infinity.  Since
$2M-2p>-1$, dominated convergence gives both prior limits
uniformly on $K\times K$.  Parseval's identity identifies the
constant-spectrum limit with Eq.~\eqref{eq:tempered-low-limit}.

It remains to establish the point-data limits, rather than apply the
homogeneous-field theorem to point masses.  In $L^2(\mathbb R,d\omega)$
put
\begin{equation*}
 h_{\lambda,z}(\omega)=\sqrt{\frac{S_{p,\kappa}(\omega)}{2\pi}}
   \sqrt{\lambda a}\,e^{-i\lambda b\omega}
       \widehat\psi(\lambda a\omega),\qquad
 g_i(\omega)=\sqrt{\frac{S_{p,\kappa}(\omega)}{2\pi}}
       e^{-it_i\omega}.
\end{equation*}
Then $q_{\lambda z,i}=\langle h_{\lambda,z},g_i\rangle_2$.
The point-data vector $g_i$ belongs to $L^2$ because $p>1/2$;
indeed $\|g_i\|_2^2=k(0)<\infty$.
The full $L^2$ norms of $\lambda^{-p}h_{\lambda,z}$ for
$\lambda\downarrow0$ and of $h_{\lambda,z}$ for
$\lambda\uparrow\infty$ are bounded uniformly on $K$, by the
same dominating spectral factors as above.  For an annulus
$E_{\delta,R}=\{\delta\leq|\omega|\leq R\}$, with
$0<\delta<R<\infty$, the wavelet's $M$ zero moments give
\begin{equation*}
 \sup_{z\in K}\|\lambda^{-p}h_{\lambda,z}
            \mathbf1_{E_{\delta,R}}\|_2
       =O_K(\lambda^{M+1/2-p})\to0
       \quad(\lambda\downarrow0).
\end{equation*}
Schwartz decay likewise gives
$\sup_{z\in K}\|h_{\lambda,z}\mathbf1_{E_{\delta,R}}\|_2
\to0$ as $\lambda\uparrow\infty$.
For any $g\in L^2$, the annular truncations
$g\mathbf1_{E_{\delta,R}}$ approximate $g$ in $L^2$ as
$\delta\downarrow0$ and $R\uparrow\infty$.  Cauchy--Schwarz,
the uniform full-norm bounds, and the annular estimates therefore
prove that $\lambda^{-p}h_{\lambda,z}$ and $h_{\lambda,z}$ converge
weakly to zero at their respective endpoints, uniformly for $z\in K$.
Taking $g=g_i$ proves Eq.~\eqref{eq:tempered-point-cross-limits}
directly.  Gaussian conditioning gives
$C_W^{\mathrm{post}}(\lambda z,\lambda z')
=C_W(\lambda z,\lambda z')
-q_{\lambda z}^{\mathsf T}A^{-1}q_{\lambda z'}$.
Because $A^{-1}$ is fixed and $n$ is finite, the normalized
correction vanishes uniformly at small scales and the unnormalized
correction vanishes uniformly at large scales.  Combining this with
the prior limits proves Eq.~\eqref{eq:tempered-posterior-limits}.
\hfill$\square$

The restriction $p>1/2$ is exactly what makes the point variance
$k(0)=(2\pi)^{-1}\int S_{p,\kappa}$ finite.  The upper bound
$p<M+1/2$ ensures integrability of $H_{2p}$ for a wavelet with $M$
vanishing moments and is sharp when its first nonzero Fourier Taylor
coefficient has order exactly $M$; at equality the integral diverges
logarithmically.  A particular wavelet with more vanishing moments
may allow a larger $p$.  The assumption $\kappa>0$ makes $S(0)$ finite.
The small-scale kernel $H_{2p}$ is the prior of the homogeneous
generalized field above with $\gamma=2p$; $H_0$ is a constant-spectrum
white-noise prior on test functions.  Neither limiting generalized
field admits the same point evaluations as the finite-variance process,
so these are limits of \emph{conditioned wavelet covariances}, not
limits obtained by conditioning either limiting field at points.
The different spectral weights make the two prior kernels distinct
on the wavelet family.  Exact finite-scale dilation symmetry is not
asserted for Eq.~\eqref{eq:tempered-spectrum}.

For a numerical illustration with no additional figure, use the
Mexican-hat wavelet in Eq.~\eqref{eq:mexican-hat} and set
$M=2$, $C=\kappa=1$, $p=1$.  Observe $y=x(0)+\varepsilon$ with
$\operatorname{Var}(\varepsilon)=1/2$; then $k(0)=1/2$ and $A=1$.
Writing $V^{\mathrm{post}}(\lambda,0)$ for its posterior variance,
the endpoint values follow from Eqs.~\eqref{eq:tempered-high-limit} and
\eqref{eq:tempered-low-limit}:
\begin{equation*}
 \lim_{\lambda\downarrow0}\lambda^{-2}
 V^{\mathrm{post}}(\lambda,0)=\frac23,\qquad
 \lim_{\lambda\uparrow\infty}V^{\mathrm{post}}(\lambda,0)=1.
\end{equation*}
Direct quadrature of the spectral variance and point cross-covariance
gives $\lambda^{-2}V^{\mathrm{post}}(\lambda,0)=0.665915$
and $0.659199$ for $\lambda=10^{-3}$ and $10^{-2}$, respectively;
$V^{\mathrm{post}}(\lambda,0)=0.904853$ and $0.992232$ for
$\lambda=10$ and $10^2$, respectively.  The point-data correction
is nonzero at finite scales and vanishes in the required normalization
at both endpoints.

\paragraph{One Mexican-hat observation: closed forms and verification.}
Take
\begin{equation}
 \psi(t)=k(1-t^2)e^{-t^2/2},\qquad
 k=\frac{2}{\sqrt3\,\pi^{1/4}},\qquad
 \widehat\psi(\omega)=k\sqrt{2\pi}\,\omega^2e^{-\omega^2/2}.
 \label{eq:mexican-hat}
\end{equation}
This real Schwartz wavelet has $\|\psi\|_2=1$ and two vanishing moments.
Hence $M=2$ and every $\gamma<5$ is admissible.  Write
$\nu=(5-\gamma)/2>0$, $B_\gamma=Ck^2\Gamma(\nu)$,
$\beta=(a^2+a'^2)/2$, and $\delta=b-b'$.
Let $\Phi_\nu(x)={}_1F_1(\nu;1/2;x)$, where
${}_1F_1(\alpha;\beta;x)=\sum_{j=0}^\infty
(\alpha)_j x^j/((\beta)_j j!)$ and $(\alpha)_j$ is the rising factorial.
Direct substitution in Eq.~\eqref{eq:crossspectral} gives
\begin{align}
 C_W(z,z')
 &=Ck^2(aa')^{5/2}\int_{\mathbb R}|\omega|^{4-\gamma}
       e^{-\beta\omega^2}\cos(\delta\omega)\,d\omega \notag\\
 &=B_\gamma(aa')^{5/2}\beta^{-\nu}
       \Phi_\nu\!\left(-\frac{\delta^2}{4\beta}\right).
 \label{eq:mexican-covariance}
\end{align}
To verify the integral, expand the cosine, use
$\int_{\mathbb R}|\omega|^{2\nu+2j-1}e^{-\beta\omega^2}\,d\omega
=\beta^{-\nu-j}\Gamma(\nu+j)$, and sum the resulting series.
Termwise integration is justified by the integrable bound
\begin{equation*}
 |\omega|^{2\nu-1}e^{-\beta\omega^2}\cosh(|\delta\omega|).
\end{equation*}
In particular,
\begin{align}
 V(a)&=B_\gamma a^\gamma,\notag\\
 C_W\bigl((a,0),(a',0)\bigr)
 &=B_\gamma(aa')^{\gamma/2}
       \left(\frac{2aa'}{a^2+a'^2}\right)^\nu.
 \label{eq:mexican-crossscale}
\end{align}
For $\gamma=0$, $B_0=C$, providing the white-noise normalization check.

Fix $a_0>0$ and observe $y=W(a_0,0)+\varepsilon$, where
$\varepsilon$ is independent centered Gaussian noise with
$0\leq\tau^2=\operatorname{Var}(\varepsilon)<\infty$.
Set
\begin{equation*}
 \eta=\frac{B_\gamma a_0^\gamma}{B_\gamma a_0^\gamma+\tau^2}
       \in(0,1],\qquad
 H_\nu(a,b)=\Phi_\nu\!\left(-\frac{b^2}{2(a^2+a_0^2)}\right).
\end{equation*}
Using Eq.~\eqref{eq:posterior} gives the closed-form posterior variance
\begin{equation}
 V^{\mathrm{post}}(a,b)=B_\gamma a^\gamma
 \left[1-\eta\left(\frac{2aa_0}{a^2+a_0^2}\right)^{5-\gamma}
              H_\nu(a,b)^2\right].
 \label{eq:mexican-posterior}
\end{equation}
At $b=0$, write $r=a/a_0$.  Then
\begin{equation}
 \frac{V(a)-V^{\mathrm{post}}(a,0)}{V(a)}
   =\eta\left(\frac{2r}{1+r^2}\right)^{5-\gamma}>0
   \qquad(0<r<\infty).
 \label{eq:mexican-reduction}
\end{equation}
The fractional reduction attains its maximum $\eta$ at $a=a_0$ and
tends to zero at both scale endpoints.  In particular,
$V^{\mathrm{post}}(a_0,0)=B_\gamma a_0^\gamma\tau^2/
(B_\gamma a_0^\gamma+\tau^2)$, including zero variance for a noiseless
observation.  By Corollary~\ref{cor:broken-symmetry}, the nonzero reduction
breaks the global dilation identity.
More generally, for fixed $a>0,b\in\mathbb R$ the normalized reduction is
\begin{equation}
 \begin{aligned}
 &\lambda^{-\gamma}\bigl[V(\lambda a)
                  -V^{\mathrm{post}}(\lambda a,\lambda b)\bigr]\\
 &\qquad=B_\gamma a^\gamma\eta
       \left(\frac{2\lambda aa_0}{\lambda^2a^2+a_0^2}\right)^{5-\gamma}
       H_\nu(\lambda a,\lambda b)^2.
 \end{aligned}
 \label{eq:mexican-limits}
\end{equation}
The argument of $\Phi_\nu$ in the last factor tends to $0$ as
$\lambda\downarrow0$ and to $-b^2/(2a^2)$ as $\lambda\uparrow\infty$.
It therefore stays in a compact interval where $\Phi_\nu$ is bounded.
Since $5-\gamma>0$, the power factor tends to zero at either endpoint.
Thus $\lambda^{-\gamma}V^{\mathrm{post}}(\lambda a,\lambda b)
\to B_\gamma a^\gamma$ in both limits, as predicted.
In this example $\phi_1=\psi_{a_0,0}$, so $m_\psi=m_1=2$ and
$\alpha=\beta_1=\nu=(5-\gamma)/2$ in
Theorem~\ref{thm:quantitative-recovery}.  Its cross-covariance
coefficients specialize to
\begin{equation*}
 S_1(a,b)=B_\gamma 2^\nu a^{5/2}a_0^{\gamma-5/2},\qquad
 L_1(a,b)=B_\gamma 2^\nu a^{\gamma-5/2}a_0^{5/2}
       \Phi_\nu\!\left(-\frac{b^2}{2a^2}\right).
\end{equation*}
Indeed, expanding the exact expression in
Eq.~\eqref{eq:mexican-limits} gives, for fixed $a>0$ and $b$,
\begin{align}
 \lambda^{-\gamma}\bigl[V(\lambda a)
       -V^{\mathrm{post}}(\lambda a,\lambda b)\bigr]
 &=B_\gamma a^\gamma\eta
    \left(\frac{2a}{a_0}\right)^{5-\gamma}
    \lambda^{5-\gamma}+O(\lambda^{7-\gamma}),
          &&\lambda\downarrow0,\notag\\
 &=B_\gamma a^\gamma\eta
    \left(\frac{2a_0}{a}\right)^{5-\gamma}
    \Phi_\nu\!\left(-\frac{b^2}{2a^2}\right)^2
    \lambda^{-(5-\gamma)}
    +O(\lambda^{-(7-\gamma)}),
          &&\lambda\uparrow\infty.
 \label{eq:mexican-asymptotics}
\end{align}
The exponent $5-\gamma$ at each endpoint is therefore the
$m_\psi=m_1=2$ case of Eq.~\eqref{eq:uniform-posterior-rate}.
The small-scale coefficient is positive; the large-scale coefficient
can vanish at positions where $\Phi_\nu(-b^2/(2a^2))=0$, whereas
at $b=0$ both coefficients are positive.

For the illustrations in Figs.~\ref{fig:fixed-observation-dilation}
and \ref{fig:scale-position-reduction}, set $C=1$, $a_0=1$, and
$y=W(1,0)+\varepsilon$, where $\varepsilon$ is centered Gaussian noise
independent of the field, with variance $\tau_\gamma^2=B_\gamma/4$.
Thus $\eta=0.8$, with the same relative noise strength for each $\gamma$;
for each exponent, the observation and its noise variance remain fixed
throughout target dilation.  These choices apply only to the figures.
The quantities $V$, $V^{\mathrm{post}}$, $C_W$, and
$C_W^{\mathrm{post}}$ below are evaluated at the indicated $\gamma$.
Figure~\ref{fig:fixed-observation-dilation}(a) shows
\begin{equation*}
 R_\gamma(\lambda)
 =\frac{V^{\mathrm{post}}(\lambda,0)}{B_\gamma\lambda^\gamma}
 =1-0.8\left(\frac{2\lambda}{1+\lambda^2}\right)^{5-\gamma}
 \qquad(\gamma=0,1,3).
\end{equation*}
All three curves satisfy $R_\gamma(1)=0.2$ and approach the prior value
$1$ at both endpoints.  Their dependence on $\lambda$ displays the
failure of exact dilation symmetry with a fixed observation.
Panel (b) uses $\gamma=3$, $z_1=(1,0.5)$, and $z_2=(2,-0.5)$ to show
\begin{equation*}
 G(\lambda)=\frac{\lambda^{-3}
 C_W^{\mathrm{post}}(\lambda z_1,\lambda z_2)}{\sqrt{V(1)V(2)}}.
\end{equation*}
Both target scales and positions are dilated simultaneously while the
observation stays fixed.  This is a covariance normalized by prior
standard deviations, not a posterior correlation coefficient; its prior
baseline is $C_W(z_1,z_2)/\sqrt{V(1)V(2)}$, with all covariances and
variances evaluated at $\gamma=3$.  The finite plotting range illustrates
the symmetry breaking and endpoint recovery proved above.

\begin{figure}[tbp]
  \centering
  \includegraphics[width=\linewidth]{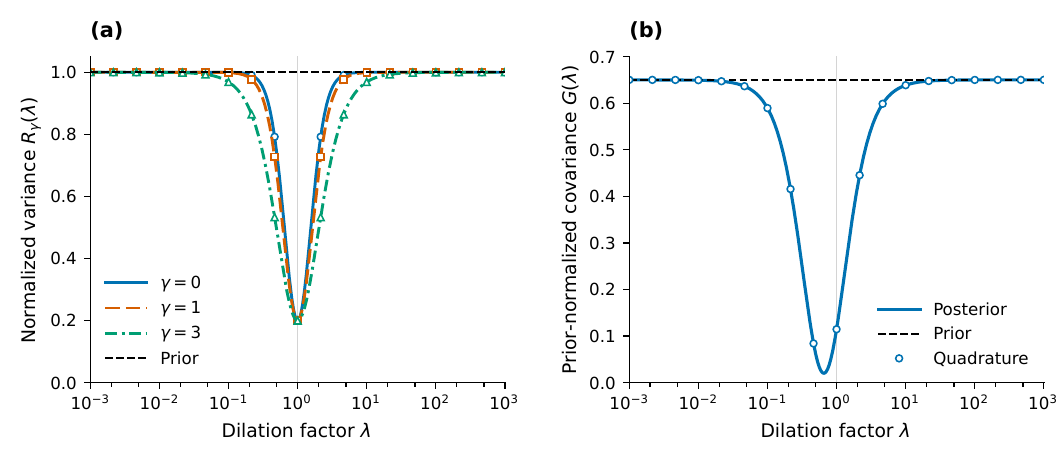}
  \caption{Fixed-observation dilation for the Mexican-hat field with
  $C=1$, $a_0=1$, and $y=W(1,0)+\varepsilon$. The centered Gaussian noise
  is independent of the field, with $\tau_\gamma^2=B_\gamma/4$ and
  $\eta=0.8$; the observation and noise variance are fixed as $\lambda$
  varies.
  (a) $R_\gamma(\lambda)=V^{\mathrm{post}}(\lambda,0)/
  (B_\gamma\lambda^\gamma)$ for $\gamma=0,1,3$ (solid, dashed, and
  dash-dotted curves, respectively).
  (b) For $\gamma=3$, $z_1=(1,0.5)$ and $z_2=(2,-0.5)$, the solid curve
  shows $G(\lambda)=\lambda^{-3}C_W^{\mathrm{post}}(\lambda z_1,
  \lambda z_2)/\sqrt{V(1)V(2)}$, with both target scales and positions
  dilated simultaneously. The normalization uses prior standard
  deviations, not posterior ones; all quantities in (b) use $\gamma=3$.
  Black dashed lines show the prior values $1$ and
  $C_W(z_1,z_2)/\sqrt{V(1)V(2)}$, respectively. Open markers show
  independent spectral quadrature, and thin vertical lines mark
  $\lambda=1$.}
  \label{fig:fixed-observation-dilation}
\end{figure}

Figure~\ref{fig:scale-position-reduction} shows the scale and position
dependence of the same observation's influence.  In local dimensionless
coordinates $r=a/a_0$ and $u=b/a_0$, define the fractional variance reduction
\begin{equation*}
 D_\gamma(r,u)
 =\frac{V(a_0r)-V^{\mathrm{post}}(a_0r,a_0u)}{V(a_0r)}
 =\eta\left(\frac{2r}{1+r^2}\right)^{5-\gamma}
 H_\nu(a_0r,a_0u)^2,
\end{equation*}
where the last equality is Eq.~\eqref{eq:mexican-posterior}.
This quantity is neither the posterior variance itself nor a correlation
coefficient.  For the plotted $a_0=1$ and $\eta=0.8$,
$D_\gamma(r,0)=1-R_\gamma(r)$, linking the heatmaps to
Fig.~\ref{fig:fixed-observation-dilation}(a).
The $\gamma=0$ panel has pronounced spatial side maxima, arising from
the squared covariance with the observation in Eq.~\eqref{eq:posterior}.
For $\gamma=3$, the central influence extends over a broader range of
scales.  Neither the displayed window nor its dark-purple regions imply
that the influence is exactly zero.

\begin{figure}[tbp]
  \centering
  \includegraphics[width=\linewidth]{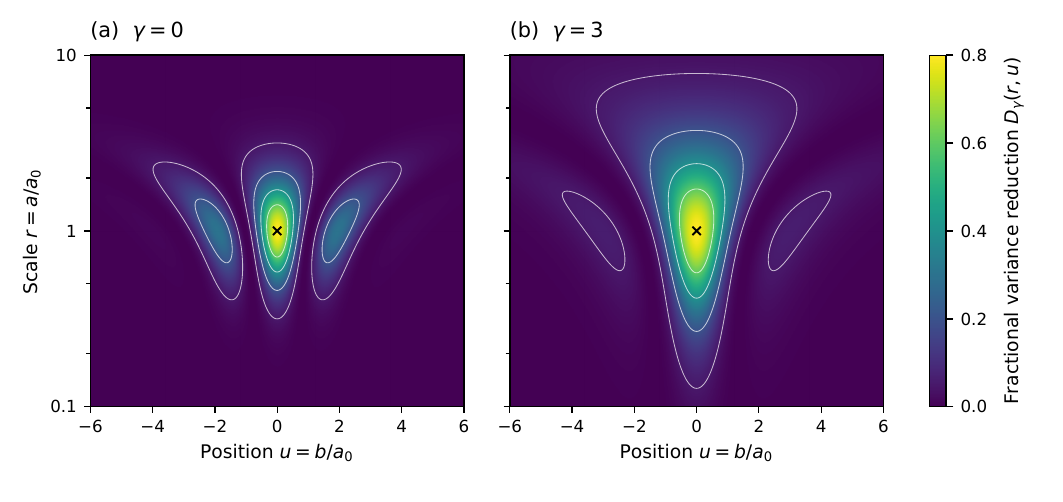}
  \caption{Scale and position dependence of fractional variance reduction
  for (a) $\gamma=0$ and (b) $\gamma=3$. The Mexican-hat field has $C=1$
  and the fixed observation $y=W(1,0)+\varepsilon$, with $a_0=1$.
  Centered Gaussian noise independent of the field has fixed variance
  $\tau_\gamma^2=B_\gamma/4$, so $\eta=0.8$ in both panels.
  The horizontal coordinate is $u=b/a_0$; the vertical coordinate
  $r=a/a_0$ is logarithmic. Color shows
  $D_\gamma(r,u)=[V(a_0r)-V^{\mathrm{post}}(a_0r,a_0u)]/V(a_0r)$,
  evaluated at each panel's $\gamma$, on a shared linear scale from
  $0$ to $0.8$. Black crosses mark $(u,r)=(0,1)$, and white contours
  show $D_\gamma=0.05,0.2,0.4,0.6$.}
  \label{fig:scale-position-reduction}
\end{figure}

As an independent numerical check, we evaluated the spectral integral
in Eq.~\eqref{eq:mexican-covariance} by adaptive quadrature, without
using its hypergeometric expression.  The substitution
$x=\sqrt\beta\,\omega$ permits algebraically weighted quadrature on
$[0,1]$ and ordinary quadrature on $[1,12]$; for the parameters below,
the omitted Gaussian tail is negligible at the stated tolerance.
With absolute and relative quadrature tolerances $2\times10^{-12}$,
we checked all 180 combinations of $C=1.7$,
$\gamma\in\{-3,0,1,3,4.8\}$, $a\in\{0.2,0.7,2,5\}$,
$a'\in\{0.3,1.3,3\}$, and $\delta\in\{0,0.8,4\}$.
The largest observed discrepancy divided by $\sqrt{V(a)V(a')}$ was
less than $10^{-14}$.
For $a_0=1.3$, $\tau^2=0.4$, and all 60 combinations of the same
$\gamma,a$ and $b\in\{0,0.8,4\}$, inserting three independently
integrated covariances into Eq.~\eqref{eq:posterior} agreed with
Eq.~\eqref{eq:mexican-posterior} to less than $10^{-14}$ of $V(a)$.
These are finite numerical checks; validity for all $\gamma<5$ follows
from the preceding integral calculation and proof.

Separate saved validation records for the figures compare the conditioned
quantities at 76 parameter points for Fig.~\ref{fig:fixed-observation-dilation}
(four curves, 19 points each) and at 70 combinations of $(\gamma,r,u)$
for Fig.~\ref{fig:scale-position-reduction}.  These are comparison-point
counts, not numbers of integral evaluations.  Both validations evaluate
the original spectral integral in Eq.~\eqref{eq:mexican-covariance}
after nondimensionalization by $x=\sqrt\beta\,\omega$, using adaptive
quadrature on $[0,\infty)$ with absolute and relative tolerances
$10^{-12}$ and no hypergeometric function in the quadrature calculation.
The saved CSV data and JSON summaries give maximum observed differences
from the closed forms of approximately $4.44\times10^{-16}$ for the
normalized quantities in Fig.~\ref{fig:fixed-observation-dilation} and
$2.78\times10^{-16}$ for $D_\gamma$ in
Fig.~\ref{fig:scale-position-reduction}.  These observed differences
are distinct from the integrator's error estimates and are not rigorous
error bounds.

\section{Discussion}
\label{sec:discussion}
\paragraph{Relation to existing results and the present contribution.}
The spectral-filter formulas, wavelet variance power laws, and
cross-scale dependence belong to established wavelet theory
\cite{wornell1992,moulines2007}.
For example, Coeurjolly, Amblard, and Achard derive wavelet covariances
and cross-spectra for multivariate fractional Brownian motion
\cite{coeurjolly2010}.  Their Theorem~3.3 (Theorem~3 in the preprint)
gives the cross-spectral density at different scales, its low-frequency
behavior, and coherence, under their wavelet integrability assumptions
and at least two vanishing moments.  It concerns the prior multivariate
fractional Brownian motion, whose Hurst exponents lie in $(0,1)$;
it does not condition on a fixed finite observation set.
Equations~\eqref{eq:crossspectral}--\eqref{eq:jointscaling} use this
established spectral-filter and homogeneous-covariance structure.
Rasmussen and Williams give the Gaussian conditional covariance for
finite training and test vectors with independent identically distributed
Gaussian observation noise \cite[Eq.~(2.24)]{rasmussen2006}.
Equation~\eqref{eq:posterior} uses the same conditioning formula for
admissible linear observations with fixed noise covariance $N$.
Sottinen and Yazigi's Theorem~3.1 gives the orthogonal representation
and finite-rank covariance subtraction for generalized Gaussian bridges,
conditioned on exact values of finitely many admissible linear
functionals of a continuous Gaussian process on a finite interval,
with an invertible observation Gram matrix
\cite{sottinen2014}.
Our admissible test-function formulation concerns a generalized field;
we therefore prove the finite-dimensional conditioning step directly
rather than assume their continuous-path hypotheses for this field.
The common projection mechanism, including the noiseless case and its
extension to independent Gaussian observation noise, is established
theory.
Marcotte and Gloaguen \cite[abstract and Appendix]{marcotte2015}
provide a closer wavelet-specific precedent: their discrete matrix
construction conditions Gaussian simulations on wavelet coefficients
and hard point data, and their appendix proves that both sets of
constraints are reproduced exactly.  Our comparison with that work
is restricted to the construction described in its published abstract
and accessible appendix.  The present results concern the normalized
posterior covariance along continuous target dilations with fixed
observations, including compact-uniform endpoint recovery and the
necessary and sufficient condition for global dilation symmetry.
Point observations here belong to the finite-variance Mat\'ern-type
model, whereas the homogeneous generalized field requires admissible
test functions.

Pathak and Pathak \cite[Theorems~3.1 and~3.3]{pathak2014} express
large-scale wavelet transforms through moments of the transformed
distribution and derivatives of the wavelet, and small-scale transforms
through wavelet moments and derivatives of the transformed function,
at fixed translation.  The Taylor/moment mechanism and the role of
vanishing coefficients are shared with our quantitative analysis.
Theorem~\ref{thm:quantitative-recovery} carries that mechanism through
the spectrally weighted cross-covariances and the quadratic Gaussian
correction, with simultaneous dilation of scales and positions and
uniform remainders on fixed compact target sets.  Its contribution is
the resulting posterior error coefficients, recovery orders, and
cancellation bounds, not moment asymptotics in themselves.

Conditioning Gaussian measures on closed subspaces of a Hilbert space
is also established theory; Owhadi and Scovel describe the conditional
covariance there by a shorted operator \cite{owhadi2015}.
Their Hilbert-valued Gaussian measure has a trace-class covariance,
whereas the isonormal coordinates on an infinite-dimensional $E$ here
do not form an $E$-valued random vector with identity covariance.
We therefore specify the sigma-field $\mathcal F_E$ using Gaussian
coordinates and derive Eq.~\eqref{eq:closed-space-posterior} directly.
Neither infinite-dimensional conditioning nor the orthogonal
projection identity is claimed as new.  Proposition~
\ref{prop:closed-observation-space} instead characterizes when the
fixed projection of the normalized wavelet dilation orbit vanishes at
each endpoint; the finite-band and high-frequency examples prove that
infinite dimension alone does not determine the answer.

The logarithmic-scale viewpoint is the Lamperti correspondence
\cite{lamperti1962,burnecki1997}.
For the present field, take $\gamma_0=\gamma$ in Section~\ref{sec:scaling} and write
$Y_s(z)=e^{-\gamma s/2}W(e^s z)$.
Equations~\eqref{eq:spectral-isometry}--\eqref{eq:dilation-intertwining} give
\begin{equation*}
 \operatorname{Cov}\bigl(Y_s(z),Y_t(z')\bigr)
 =\langle D_{e^s}h_z,D_{e^t}h_{z'}\rangle_2
 =\langle h_z,D_{e^{t-s}}h_{z'}\rangle_2.
\end{equation*}
Hence any finite vector of these centered Gaussian coefficients is
stationary in logarithmic scale, and its cross-covariances vanish as
$|t-s|\to\infty$ by Lemma~\ref{lem:weak-dilation}.
This identity is derived on positive scales and assumes no extension to
scale zero or positive value of $\gamma/2$.
The posterior recovery result expresses the disappearance of correlations
with a fixed finite set of observations in either logarithmic-scale
direction.  This is distinct from integrating out wavelet detail spaces
in a specified RG construction, as in Ref.~\cite{marchand2022}; no such
additional construction is introduced here.

The principal conclusions established in Section~\ref{sec:posterior} are
the two limits, uniform on fixed compact target sets, in
Proposition~\ref{prop:posterior-recovery} and the necessary and
sufficient symmetry condition in Corollary~\ref{cor:broken-symmetry},
supplemented by coefficient-level recovery rates under the additional
assumptions of Theorem~\ref{thm:quantitative-recovery}.  Weak convergence of normalized
dilations makes the fixed finite observation correction vanish in both
limits in Eq.~\eqref{eq:normalized-correction}.  Combining this recovery
with the prior dilation law shows that global posterior symmetry is
possible exactly when all observation test functions vanish: every
nonzero admissible observation function has a nonzero correction somewhere
on the wavelet family.  This symmetry conclusion does not require
the finite Fourier-order assumptions used for quantitative rates.
For an ordinary Mat\'ern-type process, the spectral form and related
local scaling results are established \cite{rasmussen2006,limteo2009}.
Theorem~\ref{thm:tempered-point-endpoints} adds a two-sided statement
for fixed finite noisy point data: the finite-rank posterior correction
vanishes at each endpoint, but the recovered prior kernels are
$H_{2p}$ and $H_0$, respectively.  The proof treats the point-data
cross-covariances directly; it does not claim the Mat\'ern spectrum,
the individual prior limits, or Gaussian conditioning as new results.
For exact observation of a fixed closed Gaussian subspace, Proposition~
\ref{prop:closed-observation-space} identifies projection-norm decay
as the necessary and sufficient endpoint condition.  Both endpoints
recover for an infinite-dimensional finite frequency band, but the
small-scale endpoint fails when all frequencies $|\omega|\geq1$ are
observed.  The Mexican-hat formulas and independent spectral quadrature
provide an explicit illustration of the finite-observation conclusions.

\paragraph{Interpretation, limitations, and future directions.}
Equation~\eqref{eq:normalized-correction} identifies the effect of finite
observations: their contribution to normalized covariances vanishes
uniformly on fixed compact target sets at both dilation endpoints,
even when it is nonzero at intermediate scales.
Corollary~\ref{cor:broken-symmetry} characterizes preservation
of the global covariance identity; every nonzero admissible observation
function excludes that identity without excluding every equality at a
particular pair of scales.
For the homogeneous field, recovery concerns normalized covariances
and does not assert that the unnormalized posterior covariance approaches
the prior covariance.
It is uniform on each fixed compact target set for
the stated homogeneous field,
with fixed admissible observation functions, a fixed finite observation
count, field-independent Gaussian noise, and positive-definite observation
covariance.  Exact observation of a fixed closed subspace has the
separate criterion in Proposition~\ref{prop:closed-observation-space};
its high-frequency example shows that two-sided recovery need not hold
for infinite-dimensional observations.  No uniform limit over the full
target domain or operator-norm limit is proved; the coefficient-level
rates require the additional
finite Fourier-order assumptions of
Theorem~\ref{thm:quantitative-recovery}.  Point observations of this
homogeneous field remain undefined; Theorem~\ref{thm:tempered-point-endpoints}
instead covers the particular finite-variance spectrum in
Eq.~\eqref{eq:tempered-spectrum}.  Arbitrary ordinary Gaussian processes
beyond that model, non-Gaussian observation noise, and observation sets
that change with dilation are outside the results.
Future work may examine nonstationary models or observation families
that vary with scale.  A genuine RG construction would additionally
require specified elimination of degrees of freedom and composition
across scales; neither extension is carried out here.

\section{Conclusion}
For the centered homogeneous generalized Gaussian field on $\mathcal S_M$
with spectral covariance weight $C|\omega|^{-\gamma}$ and
$\gamma<2M+1$, Proposition~\ref{prop:posterior-recovery} proves recovery
of the prior wavelet covariance at both dilation endpoints after
multiplication of the posterior covariance by $\lambda^{-\gamma}$,
with target scales and positions dilated simultaneously.  The recovery
is uniform when both targets range over any fixed compact subset of
$(0,\infty)\times\mathbb R$.
The observations are fixed finite admissible linear functionals with
Gaussian noise independent of the field and positive-definite total
observation covariance.  Under these conditions,
Corollary~\ref{cor:broken-symmetry} proves that the posterior satisfies
the original dilation identity for every target pair and every positive
dilation factor if and only if the observation correction is identically
zero on the wavelet family, equivalently if all observation test functions
vanish.  Any nonzero admissible observation function therefore breaks
this global symmetry.
When there is at least one nonzero observation function and the Fourier
transforms have finite first nonzero orders,
Theorem~\ref{thm:quantitative-recovery} gives their leading
cross-covariance coefficients and compact-uniform posterior recovery
orders $2m_\psi+1-\gamma$ and $2\min_i m_i+1-\gamma$, with explicit
qualification when the leading coefficients vanish.

For the finite-variance Mat\'ern-type spectrum
$C(\kappa^2+\omega^2)^{-p}$ with $1/2<p<M+1/2$,
Theorem~\ref{thm:tempered-point-endpoints} establishes a different
two-endpoint result under fixed finite noisy point observations.
After multiplication by $\lambda^{-2p}$ as $\lambda\downarrow0$,
the posterior wavelet covariance converges to the homogeneous
$|\omega|^{-2p}$ prior kernel.  Without normalization as
$\lambda\uparrow\infty$, it converges to the distinct
constant-spectrum $S(0)$ prior kernel.  The point-data cross-covariances
vanish in the respective normalizations, and both covariance limits
are uniform on fixed compact target sets.  The point observations
belong to the finite-variance process, not to either limiting
generalized field.

For exact observations of an arbitrary fixed closed Gaussian subspace,
Proposition~\ref{prop:closed-observation-space} replaces the finite
matrix correction by orthogonal projection and proves that endpoint
recovery holds exactly when $\|P_ED_\lambda h_z\|_2\to0$ for every
target $z$; this criterion also gives compact-uniform recovery.
The symmetric finite band $1\leq|\omega|\leq2$ is
infinite-dimensional yet recovers at both endpoints.  Observation of
all $|\omega|\geq1$ gives a zero normalized posterior kernel as
$\lambda\downarrow0$ and recovers the prior as
$\lambda\uparrow\infty$.  The projection formula itself is standard
Gaussian conditioning; the endpoint criterion and these examples
identify the role of finite observation count in the earlier result.

The Mexican-hat formulas
\eqref{eq:mexican-covariance}--\eqref{eq:mexican-asymptotics} exhibit this
combination of symmetry breaking and normalized covariance recovery for
$\gamma<5$, with the exponent $5-\gamma$ at both endpoints as the
$m_\psi=m_1=2$ case of the general result.
At the observation position, the fractional variance
reduction is maximal at the observed scale and vanishes at both scale
endpoints.  These formulas illustrate the general theorems; Gaussian
conditioning, projection identities, and Taylor/moment expansions remain
the established tools on which those theorems build.
The limits concern fixed compact target sets, not the full target domain,
and do not cover changing observation families or arbitrary ordinary
Gaussian processes beyond Eq.~\eqref{eq:tempered-spectrum}.
The dilation identities express covariance scaling, not scale
independence or an RG construction.


\begin{thebibliography}{10}
\bibitem{wornell1992}
G.~W. Wornell and A.~V. Oppenheim,
``Estimation of fractal signals from noisy measurements using wavelets,''
\emph{IEEE Transactions on Signal Processing} \textbf{40}, 611--623 (1992).

\bibitem{moulines2007}
E. Moulines, F. Roueff, and M.~S. Taqqu,
``On the spectral density of the wavelet coefficients of long-memory
time series with application to the log-regression estimation of the
memory parameter,'' \emph{Journal of Time Series Analysis}
\textbf{28}, 155--187 (2007); arXiv:math/0512635.

\bibitem{coeurjolly2010}
J.-F. Coeurjolly, P.-O. Amblard, and S. Achard,
``Wavelet analysis of the multivariate fractional Brownian motion,''
\emph{ESAIM: Probability and Statistics} \textbf{17}, 592--604 (2013);
doi:10.1051/ps/2012011; arXiv:1007.2109.

\bibitem{marchand2022}
T. Marchand, M. Ozawa, G. Biroli, and S. Mallat,
``Multiscale Data-Driven Energy Estimation and Generation,''
\emph{Physical Review X} \textbf{13}, 041038 (2023);
doi:10.1103/PhysRevX.13.041038;
preprint: ``Wavelet Conditional Renormalization Group,''
arXiv:2207.04941 (2022).

\bibitem{rasmussen2006}
C.~E. Rasmussen and C.~K.~I. Williams,
\emph{Gaussian Processes for Machine Learning},
MIT Press, Cambridge, MA (2006), Chapters~2 and~4.

\bibitem{sottinen2014}
T. Sottinen and A. Yazigi,
``Generalized Gaussian bridges,''
\emph{Stochastic Processes and their Applications}
\textbf{124}, 3084--3105 (2014);
doi:10.1016/j.spa.2014.04.002; arXiv:1205.3405.

\bibitem{lamperti1962}
J. Lamperti,
``Semi-stable stochastic processes,''
\emph{Transactions of the American Mathematical Society}
\textbf{104}, 62--78 (1962).

\bibitem{burnecki1997}
K. Burnecki, M. Maejima, and A. Weron,
``The Lamperti transformation for self-similar processes,''
\emph{Yokohama Mathematical Journal} \textbf{44}, 25--42 (1997).

\bibitem{limteo2009}
S.~C. Lim and L.~P. Teo,
``Generalized Whittle--Mat\'ern random field as a model of correlated
fluctuations,'' \emph{Journal of Physics A: Mathematical and Theoretical}
\textbf{42}, 105202 (2009); doi:10.1088/1751-8113/42/10/105202.

\bibitem{owhadi2015}
H. Owhadi and C. Scovel,
``Conditioning Gaussian measure on Hilbert space,''
arXiv:1506.04208 (2015).

\bibitem{marcotte2015}
D. Marcotte and E. Gloaguen,
``Exact Conditioning of Gaussian Fields on Wavelet Coefficients,''
\emph{Mathematical Geosciences} \textbf{47}, 277--300 (2015);
doi:10.1007/s11004-014-9552-z.

\bibitem{pathak2014}
R.~S. Pathak and A. Pathak,
``Moment Asymptotic Expansions of the Wavelet Transforms,''
arXiv:1405.0146 (2014).
\end{thebibliography}
\end{document}